\documentclass[12pt]{article}

\usepackage{amsmath}
\usepackage{amsfonts}
\usepackage{amsthm}

\begin{document}


\author{Attila Losonczi}
\title{Extending means to several variables}

\date{30 Oct 2018}

\newtheorem{thm}{\qquad Theorem}[section]
\newtheorem{prp}[thm]{\qquad Proposition}
\newtheorem{lem}[thm]{\qquad Lemma}
\newtheorem{cor}[thm]{\qquad Corollary}
\newtheorem{rem}[thm]{\qquad Remark}
\newtheorem{ex}[thm]{\qquad Example}
\newtheorem{prb}[thm]{\qquad Problem}
\newtheorem{df}[thm]{\qquad Definition}

\maketitle

\begin{abstract}

\noindent

We begin the study of how to extend few variable means to several variable ones and how to shrink means of several variables to less variables. With the help of one of the techniques we show that it is enough to check an inequality between two quasi-arithmetic means in 2-variables and that simply implies the inequality in m-variables. The technique has some relation to Markov chains. This method can be applied to symmetrization and compounding means as well.

\noindent
\footnotetext{\noindent
AMS (2010) Subject Classifications: 26E60, 39B12 \\

}

\end{abstract}

\section{Introduction}
In this paper we are going to study the ways of extensions of an $n$-variable mean to an $m$-variable mean ($n<m$) and vice versa shrinking an $m$-variable mean to an $n$-variable mean.

The origin of the problem was raised by M. Hajja in \cite{hajja} Problem 14: is there a natural way of deriving the definition of the $n$-variable arithmetic mean from the definition of the 2-variable version. And what can one say in general? Can we define when an $n$-variable mean is concordant to an $m$-variable mean i.e. they are the different variable versions of the same mean (where "mean" in this last context is just a variableless generic notion). Can we go that far? We start to answer these questions by presenting both positive and negative results.

On basic facts on means the reader has to consult \cite{bullen}. However we provide some basic definitions.

A n-variable mean $K$ is called strictly internal if \[\min\{a_1,\dots,a_n\}<K(a_1,\dots,a_n)<\max\{a_1,\dots,a_n\}\] provided that the set $\{a_1,\dots,a_n\}$ has at least two distinct elements. 
An n-variable mean $K$ is said to be monotone if $a_i\leq b_i\ (1\leq i\leq n)$ implies that $K(a_1,\dots ,a_n)\leq K(b_1,\dots ,b_n)$. $K$ is called continuous if $a^{(i)}_k\to c^{(i)} (1\leq i\leq n)$ then $K(a^{(1)}_k,\dots,a^{(n)}_k)\to K(c^{(1)},\dots,c^{(n)})$ i.e. $K$ is continuous as an n-variable function. $K$ is symmetric if $K(a_1,\dots ,a_n)=K(a_{p(1)},\dots ,a_{p(n)})$ for all permutations $p:\{1,\dots n\}\to\{1,\dots n\}$.

All means considered in this paper are symmetric, strictly internal, monotone and continuous if we do not say otherwise. 

Sometimes we will denote a 2-variable mean by a $\circ$, so instead of $K(a,b)$ we will write $a\circ b$.

\begin{df}\label{d1}A 2-variable mean $\circ$ is called round if it fulfills functional equation $(a\circ k)\circ (k\circ b)=k$ for all $a,b\ (a<b)$ where $k=a\circ b$.
\end{df}

\subsection{Some basic observations}
Unfortunately we cannot expect one generic, unique way to extend/shrink a mean such that it keeps concordance.
E.g. let us consider the following three 3-variable means defined on $\mathbb{R}^+$.

$K_1(a,b,c)=\sqrt[3]{abc}$

$K_2(a,b,c)=\sqrt{ab+ac+bc \over{3}}$

$K_3(a,b,c) = {\sqrt{ab}+\sqrt{ac}+\sqrt{bc} \over{3}}$

For all three means we may expect $K(a,b)=\sqrt{ab}$ being the corresponding 2-variable mean. But when we extended K, we cannot expect to get all three 3-variable means, or better to say there should be three extending methods at least. And in the opposite when we reduce the means $K_1, K_2, K_3$ to 2-variable means, we cannot expect one generic way. 

\section{Extensions}

We are going to extend means from n-variable to m-variable where $2\leq n<m$.

Let $m$ real numbers be given. We are going to describe a kind of recursive method when we create $m$ sequences from them in a way that a new element of a sequence is based on the $n$-mean of some of the previous step sequence elements and always from the same ones.

In order to describe such generic method we need some definitions first.

\begin{df}Let $I_m=\{1,\dots,m\}\ (m\in\mathbb{N})$. 

$I_m^n=\{(j_1,\dots,j_n):j_i\in I_m\}\ (n\in\mathbb{N},n<m)$.

If $t=(j_1,\dots,j_n)$ then by writing $j\in t$ we will mean that $j$ is one of its coordinates of $t$ i.e. in this context we think of $t$ as a set $\{j_1,\dots,j_n\}$. 

A partial order on $I^n_m$ is defined by $(j_1,\dots,j_n)\leq(k_1,\dots,k_n)\Longleftrightarrow j_1\leq k_1,\dots,j_n\leq k_n$.\qed
\end{df}

\begin{df}Let $n<m$. Let a $n$-variable mean $K$ be given and $a^{(1)},\dots,a^{(m)} \in \mathbb{R},\ (a^{(1)},\dots,a^{(m)})\in Dom\ K$ with $a^{(1)}\leq\dots \leq a^{(m)}$ . Let a system $T=\{t_1,\dots,t_m\}$ be given as well where $t_i\in I^n_m\ (1\leq i\leq m)$. Let us use the notation: $t_i=(j_{i,1},\dots,j_{i,n})$.

Now we define $m$ sequences in the following way.

Let $a^{(1)}_0=a^{(1)},\dots,a^{(m)}_0=a^{(m)}$.

Set $a^{(i)}_{k+1}=K(a^{(j_{i,1})}_k,\dots,a^{(j_{i,n})}_k)\ (1\leq i\leq m,k\in\mathbb{N})$.\qed
\end{df}

We also prefer the following four properties of $T$:

(1) $t_1\leq t_2\leq\dots\leq t_m$

(2) $\forall k\in I_m,\ |\{i:k\in t_i\}|=n$

(3) $\forall i\ \min t_i\leq i\leq\max t_i$; if $2\leq i\leq m-1$ then $\min t_i< i<\max t_i$

(4) $\forall i\geq 2\ \exists j<i$ such that $i\in t_j$.

\begin{df}$T=\{t_1,\dots,t_m\}$ is called admissible if it satisfies properties (1), (2), (3), (4).
\end{df}

\begin{thm}\label{t4} If $T$ is admissible for $(n,m)$ then all sequences $(a^{(i)}_k)$ converges to the same limit that is between the minimum and maximum of the underlying points.
\end{thm}
\begin{proof}Property (3) gives that $1\in t_1, m\in t_m$ and (1) implies that if $i\leq j$ then $\forall k\ a^{(i)}_k\leq a^{(j)}_k$.

Moreover $\forall k\ a^{(1)}\leq a^{(1)}_k\leq a^{(2)}_k\dots\leq a^{(m)}_k\leq a^{(m)}$ and $a^{(1)}_k$ is increasing and $a^{(m)}_k$ is decreasing hence both converges. We show that all converges to the same limit. Let $a^{(1)}_k\to c$. Suppose there is $i\geq 2$ such that $(a^{(i)}_k)$ does not converge to $c$. Let $i$ denote the least such index. Then by (4) there is $j<i$ such that $i\in t_j=(u_1,\dots,u_n)$. All sequences $(a^{(p)}_k)$ are bounded $(p\in t_j)$ hence we can find a subsequence of $(k)$ say $(k_q)$ such that all sequences $(a^{(p)}_{k_q})$ are convergent, say $(a^{(p)}_{k_q})\to w_p$. Obviously $w_p\geq c$. Let us choose $(k_q)$ such that $(a^{(i)}_{k_q})\to w_i\neq c$. By assumption $(a^{(j)}_{k_q})\to c$. By definition $(a^{(j)}_{k+1})=K(a^{(u_1)}_k,\dots,a^{(u_n)}_k)$. $K$ being strictly internal gives that $c<K(w_{u_1},\dots,w_{u_n})$. Let 
\[\epsilon=\frac{K(w_{u_1},\dots,w_{u_n})-c}{2}.\] 
$K$ is continuous therefore there exists $\delta>0$ such that $w'_{u_r}\in (w_{u_r}-\delta,w_{u_r}+\delta)\ (1\leq r\leq n)$ implies that 
\[K(w'_{u_1},\dots,w'_{u_n})\in (K(w_{u_1},\dots,w_{u_n})-\epsilon,K(w_{u_1},\dots,w_{u_n})+\epsilon).\] 
There is $N\in\mathbb{N}$ such that $k\geq N$ implies $a^{(j)}_k\in (c-\epsilon,c+\epsilon)$ and $q\geq N$ implies $a^{(p)}_{k_q}\in (w_{u_r}-\delta,w_{u_r}+\delta)\ (1\leq r\leq n)$. We get that 
\[a^{(j)}_{k_N+1}=K(a^{(u_1)}_{k_N},\dots,a^{(u_n)}_{k_N}) \in \big(K(w_{u_1},\dots,w_{u_n})-\epsilon,K(w_{u_1},\dots,w_{u_n})+\epsilon\big)\cap\big(c-\epsilon,c+\epsilon\big),\] which is a contradiction.
\end{proof}

\begin{df}If $K,\ a^{(1)},\dots,a^{(m)}$ are given, $T=T_{n,m}$ is admissible $(n<m)$ then let us denote the common limit point of the sequences $(a^{(i)}_k)$ by $K^{(T)}(a^{(1)},\dots,a^{(m)})$ which is a mean of $a^{(1)},\dots,a^{(m)}$.
\end{df}

\begin{rem}If $k$ is fixed, $a^{(i)}_k\ (1\leq i\leq m)$ can be considered as a function of $a^{(1)},\dots,a^{(m)}$ i.e. an $m$-variable function. We will use the notation $a^{(i)}_k\big(a^{(1)},\dots,a^{(m)}\big)$. 
\end{rem}

\begin{prp}\label{r2b}If $k\in\mathbb{N},\ 1\leq i\leq m,\ K$ is a continuous (monotone) mean then the function $a^{(i)}_k(x_1,\dots,x_m)$ is continuous (monotone) as well. 
\end{prp}
\begin{proof} The statement for both attributes can be shown by induction on $k$.
\end{proof}

\begin{rem}If $k$ is fixed, $a^{(i)}_k(x_1,\dots,x_m)\ (1\leq i\leq m)$ can be considered as an m-variable mean. However we are not going to discuss such means because they are not natural enough.
\end{rem}

\begin{thm}\label{tktc} If $K$ is strictly internal, monotone, continuous and $T$ is admissible then $K^{(T)}$ is strictly internal, monotone and continuous.
\end{thm}
\begin{proof} Let $K$ be strictly internal and $a^{(1)}\leq\dots\leq a^{(m)}$, moreover let $a^{(1)}<a^{(m)}$ hold. First we are going to show that $a^{(1)}<a^{(1)}_k,\ a^{(m)}_k<a^{(m)}$ hold for some $k\in\mathbb{N}$. We show it for the first inequality, the second is similar.

Assume the contrary: $\forall k\ a^{(1)}=a^{(1)}_k$. Now let us examine the points $a^{(1)}_k,\dots,a^{(m)}_k$ and let $l_k$ be the greatest index for which $a^{(1)}=a^{(l_k)}_k\ (k\in\mathbb{N}\cup\{0\})$. If $l_k=1$ for some $k$ then obviously $a^{(1)}<a^{(1)}_{k+1}$ that is a contradiction. However if $l_k>1$ then by property (3) $a^{(1)}<a^{(l_k)}_{k+1}$. Then we get that $l_{k+1}<l_k$. Hence $l_k=1$ has to hold for some $k$ and we get a contradiction.

Property (3) gives that $1\in t_1,m\in t_m$ hence $(a^{(1)}_k)$ is increasing and $(a^{(m)}_k)$ is  decreasing. Which yield strict internality of $K^{(T)}$.

\smallskip

If $K$ is monotone then $a^{(i)}\leq b^{(i)}\ (1\leq i\leq m)$ implies that $\forall k\ a^{(i)}_k\leq b^{(i)}_k$ where $(a^{(i)}_k)$ are the sequences belonging to the points $a^{(1)},\dots,a^{(m)}$, while $(b^{(i)}_k)$ are the sequences belonging to the points $b^{(1)},\dots,b^{(m)}$. Now if $a^{(i)}_k\to a,b^{(i)}_k\to b$ then $a\leq b$ hence $K^{(T)}$ is monotone.

\smallskip

In order to prove that $K^{(T)}$ is continuous let $a^{(1)}\leq\dots\leq a^{(m)}$ be given and let $p=K^{(T)}(a^{(1)},\dots,a^{(m)})$. For $\epsilon>0$ we can find $N$ such that $k>N$ implies that $p-\epsilon< a^{(1)}_k\leq p\leq a^{(m)}_k< p+\epsilon$. By Proposition \ref{r2b} $a^{(1)}_N(a^{(1)},\dots,a^{(m)}),\ a^{(m)}_N(a^{(1)},\dots,a^{(m)})$ are both continuous functions of $a^{(1)},\dots,a^{(m)}$ hence there is $\delta>0$ such that if $\forall i\ b^{(i)}\in [a^{(i)}-\delta,a^{(i)}+\delta]$ then 
\[p-\epsilon\leq a^{(1)}_N(b^{(1)},\dots,b^{(m)})\leq a^{(m)}_N(b^{(1)},\dots,b^{(m)})\leq p+\epsilon\] (the intervals are closed deliberately). If $k>N$ then 
\[p-\epsilon\leq a^{(1)}_N(a^{(1)}-\delta,\dots,a^{(m)}-\delta)\leq a^{(1)}_k(a^{(1)}-\delta,\dots,a^{(m)}-\delta)\leq a^{(1)}_k(b^{(1)},\dots,b^{(m)})\leq\]
\[\leq a^{(m)}_k(b^{(1)},\dots,b^{(m)})\leq  a^{(m)}_k(a^{(1)}+\delta,\dots,a^{(m)}+\delta)\leq a^{(m)}_N(a^{(1)}+\delta,\dots,a^{(m)}+\delta)\leq p+\epsilon\] 
where we used that $(a^{(1)}_k)$ is increasing, $(a^{(m)}_k$) is decreasing and $a^{(1)}_k,a^{(m)}_k$ are monotone. Hence we can conclude that $K^{(T)}$ is continuous.
\end{proof}

\begin{thm}There exists an admissible $T$ for $(n,m)\ (2\leq n<m,\ n,m\in\mathbb{N})$.
\end{thm}
\begin{proof}
We show a way how to construct such $T$. 
First we present an admissible $T$ for n=2 i.e. for $(2,m)$:
Let $t_1=(1,2),\ t_k=(k-1,k+1)\ (2\leq k \leq m-1),\ t_m=(m-1,m)$. One can readily check that it has properties (1),(2),(3),(4).

Then we go on by recursion on $n$. 
Let us suppose we have an admissible $T=\{t_1,\dots,t_{m-1}\}$ for $\big((n-1),(m-1)\big)$ and we construct $T'$ for $(n,m)$. Let $t_i=(j_{i,1},\dots,j_{i,n-1})\ (1\leq i\leq m-1)$. Let us define $t'_i\in T'\ (1\leq i\leq m)$
\[t'_i=\begin{cases}
(1,j_{i,1}+1,\dots,j_{i,n-1}+1) & \text{if }1\leq i\leq n-1\\
(i-(n-1),j_{i,1}+1,\dots,j_{i,n-1}+1) & \text{if } n\leq i\leq m-1\\
(m-n+1,m-n+2,\dots,m) & \text{if } i=m
\end{cases}\]

We now show that $T'$ is admissible.

Obviously $\forall i\ t'_i\in I^n_m$.

In the definition of $t'_i$ let us call the elements of $I^n_m$ in the first line: type 1, in the second line: type 2, in the third line: type 3 elements. 

(1): If $T$ satisfies (1) then so does $T'$ using also that $t'_m$ is the greatest element in $I_m^n$.

(2): If $i=1$ it is clear. 

If $2\leq i\leq m-n$ then we know that $|\{h:i-1\in t_h\in T\}|=n-1$. Take the type 2 element that starts with $i$. With that element we get $|\{h:i\in t'_h\in T'\}|=n$. 

If $m-n+1\leq i$ then the type 3 element will provide the missing point.

(3): It is easy to check for type 1, 2 and 3 elements.

(4): If $i=2$ then $t'_1$ will satisfies the condition. If $3\leq i\leq m$ is given, we know that there is $j\in I_{m-1}, j<i-1$ such that $(i-1)\in t_j$. Hence $i\in t'_j$.
\end{proof}

The following theorem gives that the $n$-variable quasi-arithmetic means are concordant in this way.

\begin{thm}\label{th7}For a quasi-arithmetic $n$-variable mean $K$, $K^{(T_{n,m})}$ is the associated $m$-variable quasi-arithmetic mean.
\end{thm}
\begin{proof}If $K$ is quasi-arithmetic than there is a strictly monotone, continuous function $f$ such that 
\[K(b^{(1)},\dots,b^{(n)})=f^{-1}\Big(\frac{f(b^{(1)})+\dots+f(b^{(n)})}{n}\Big).\]

Let $a^{(1)},\dots,a^{(m)}$ be given.

Let $T=\{t_1,\dots,t_m\},\ t_i=(j_{i,1},\dots,j_{i,n})$, $a^{(i)}_{0}=a^{(i)},\ a^{(i)}_{k+1}=a^{(i)}_{k+1}(T)=K\big(a^{(j_{i,1})}_k,\dots,a^{(j_{i,n})}_k\big)\ \ \ (1\leq i\leq m,\ k\in\mathbb{N}\cup\{0\})$.

First we show that there exist coefficients $s^{i,k}_l\ (1\leq i,l\leq m,\ k\in\mathbb{N})$ such that 
\[a^{(i)}_{k}=f^{-1}\big(s^{i,k}_1 f(a^{(1)})+\dots+s^{i,k}_m f(a^{(m)})\big)\]
and $0\leq s^{i,k}_l\leq 1$, $\forall i\forall k\ s^{i,k}_1+\dots+s^{i,k}_m=1$ holds. Clearly 
\[s^{i,1}_l=
\begin{cases}\frac{1}{n}&\text{if }l\in t_i\\ 
0&\text{otherwise}
\end{cases}\]
We go on by induction and suppose the assertion is true for $k$. 
\[a^{(i)}_{k+1}=f^{-1}\Big(\frac{f(a^{(j_{i,1})}_k)+\dots+f(a^{(j_{i,n})}_k)}{n}\Big)=\]
\begin{equation}\label{eq1}f^{-1}\Big(\frac{s^{j_{i,1},k}_1 f(a^{(1)})+\dots+s^{j_{i,1},k}_m f(a^{(m)})+\dots+s^{j_{i,n},k}_1 f(a^{(1)})+\dots+s^{j_{i,n},k}_m f(a^{(m)})}{n}\Big).\end{equation}
In the numerator if we calculate the coefficient of $f(a^{(i)})$ it will be non-negative and the sum of all coefficients will be $n$ which altogether give the statement.

Now we going to express those factors $s^{i,k}_l$ in a useful way. 
First let us define the following $m \times m$ matrix M: 

\begin{equation}\label{eq2}
M_{i,l}=\begin{cases}
\frac{1}{n} & \text{if } l\in t_i\\
0 & \text{otherwise}
\end{cases}.
\end{equation}
Clearly $s^{i,1}_l=M_{i,l}$.
We show by induction that $s^{i,k}_l=(M^k)_{i,l}$ when $k\in\mathbb{N}$. Suppose it is true for $k$. By equation (\ref{eq1}) we get that in $a^{(i)}_{k+1}$ the coefficient of $f(a^{(l)})$ equals to 
\[s^{i,k+1}_l=\frac{s^{j_{i,1},k}_l+\dots+s^{j_{i,n},k}_l}{n}=\frac{(M^k)_{j_{i,1},l}+\dots+(M^k)_{j_{i,n},l}}{n}=\big(M\cdot (M^k)\big)_{i,l}.\]

Our next aim is to prove that $\forall i \forall j\ \lim\limits_{k\to\infty}s^{i,k}_j\to \frac{1}{m}$ that would prove the theorem completely since $f^{-1}$ is continuous.

Let us have a stationary Markov chain $(X_k)_{k\in\mathbb{N}}$ with states $1,\dots,m$ and with transition matrix $M$ defined in equation (\ref{eq2}). In the theory of Markov chains there is a theorem that says that for an irreducible, aperiodic, positive recurrent and doubly stochastic Markov chain with $m$ states it holds that 
$\lim\limits_{k\to\infty} P(X_k=l)=\frac{1}{m}$ for any state $l$. This would prove the our theorem since it means that $\lim\limits_{k\to\infty}s^{i,k}_l=\lim\limits_{k\to\infty}(M^k)_{i,l}=\frac{1}{m}$ for all $i,l$.

Therefore we only have to show that $M$ has all required properties:

By property (2) of $T$, M is doubly stochastic.

M is irreducible since there is only one communication class because state "$1$"  and state "$j$" communicate ($\forall j>1$)  by property (4) of $T$.

Aperiodic: By property (3) of $T$, $t_1=(1,\dots)$ therefore $p^{(1)}_{11}>0$ hence for state "1" the period is 1 and all states in a communication class have the same period.

Positive recurrent: An irreducible finite-state Markov chain is always positive recurrent. 
\end{proof}

We can also answer one of the questions of Hajja, namely: is there a natural way to derive the $n$-variable arithmetic mean from the 2-variable arithmetic mean? Our method just provides that (use Theorem \ref{th7} with $f(x)=x$).

\begin{ex}Property (2) cannot be abandoned if we want to keep Theorem \ref{th7} valid.
\end{ex}
\begin{proof}Let $n=2,m=4,\ K(a,b)=\frac{a+b}{2},\ t_1=(1,2),t_2=(1,3),t_3=(1,4),t_4=(3,4)$. One can readily check that properties (1),(3),(4) are satisfied, (2) is not. Let $a^{(1)}=0,a^{(2)}=1,a^{(3)}=1,a^{(4)}=1$. Easy calculation shows that $a^{(4)}_4=0.6875<\frac{0+1+1+1}{4}=0.75$ and because $(a^{(4)}_k)$ is decreasing, $K^{(T)}$ is not the 4-variable arithmetic mean.
\end{proof}

\begin{ex}Properties (1),(2),(3),(4) do not imply that $T$ is unique i.e. for given pair $(n,m)$ there can be more than one such system.
\end{ex}
\begin{proof}Let $n=3,m=5$.

System 1: $t_1=(1,2,3),\ t_2=(1,2,4),\ t_3=(1,3,5),\ t_4=(2,4,5),\ t_5=(3,4,5)$.

System 2: $t_1=(1,2,4),\ t_2=(1,3,4),\ t_3=(1,3,5),\ t_4=(2,3,5),\ t_5=(2,4,5)$.
\end{proof}

\begin{prp}For $n=2$ there is a unique system $T$ for $(n,m)$ which satisfies properties (1),(2),(3),(4).
\end{prp}
\begin{proof}By (3),(4) $t_1=(1,2)$. By (2) there is $t_i=(1,s)$ with some $s$. By (2) $t_1,t_i$ are the only elements containing 1. By (1) $i=2$. By (4) $s=3$. By (4) $4\in t_3$. By (2),(3) $2\in t_3$ i.e. $t_3=(2,4)$. We can go on by induction and get $t_i=(i-1,i+1)\ (2\leq i\leq m-1)$. By (1) we get that $t_m$ has to be $(m-1,m)$.
\end{proof}

\begin{prp}\label{pnmine}If $K_1, K_2$ are two n-variable means and $K_1\leq K_2$ then $K_1^{(T_{n,m})}\leq K_2^{(T_{n,m})}$.
\end{prp}
\begin{proof} The associated sequences satisfy the same inequality.
\end{proof}

Now we can formulate one of our main results namely that an inequality between quasi-arithmetic means is enough to check in 2 variables only.

\begin{thm}\label{tm2}If $K_1,K_2$ are n-variable quasi-arithmetic means and $K_1\leq K_2$ holds in n variables then $K^{(m)}_1\leq K^{(m)}_2$ holds as well where $K^{(m)}_i$ denotes the associated m-variable quasi-arithmetic mean $( n<m)$.
\end{thm}
\begin{proof} Theorem \ref{th7} and Proposition \ref{pnmine}.
\end{proof}

\begin{prp}\label{tfsnqc}If $a^{(1)}<a^{(m)},\ \max t_1\leq\min t_m$ then \newline$\forall k\ a^{(1)}_k < K^{(T)}(a^{(1)},\dots,a^{(m)}) < a^{(m)}_k$ i.e. $(a^{(1)}_k),(a^{(m)}_k)$ are not quasi-constant.
\end{prp}
\begin{proof}We show it for $(a^{(1)}_k)$, the other is similar. 

Suppose indirectly that $a^{(1)}_k=p=K^{(T)}(a^{(1)},\dots,a^{(m)})$ if $k\geq N$. We will show by induction on $i$ that $a^{(i)}_{k}=p$ if $k\geq N,\ 1\leq i\leq m$.

Assume that $a^{(j)}_{k}=p$ if $j\leq i<m,\ k\geq N$. Let $k\geq N$ be fixed. Then $a^{(i)}_{k+1}=p=K(a^{(j_{i,1})}_{k},\dots,a^{(j_{i,n})}_{k})$. By property (3) there are $h,l$ such that $j_{i,h}\leq i<j_{i,l}$. By induction $a^{j_{i,h}}_k=p$, therefore by strict internality of $K$, all other terms have to be equal to $p$ as well, e.g. $a^{j_{i,l}}_k=p$. But we know that $a^{(j_{i,1})}_{k}\leq\dots\leq a^{(j_{i,n})}_{k}$ i.e. $p=a^{(j_{i,h})}_{k}\leq\ a^{(i+1)}_{k}\leq\ a^{(j_{i,l})}_{k}=p$ which gives that $a^{(i+1)}_{k}=p$ too.

Now let $k$ be chosen such that $a^{(1)}_{k+1}=\dots=a^{(m)}_{k+1}=p$ and $a^{(1)}_{k}\ne p,a^{(m)}_{k}\ne p$. This means that $a^{(1)}_{k}<p<a^{(m)}_{k}$. Now $a^{(1)}_{k+1}=K(a^{(j_{1,1})}_{k},\dots,a^{(j_{1,n})}_{k}),\ a^{(m)}_{k+1}=K(a^{(j_{m,1})}_{k},\dots,a^{(j_{m,n})}_{k})$. By property (3) $j_{1,1}=1,\ j_{m,n}=m$ which by strict internality of $K$ yields that $a^{(j_{1,n})}_{k}>p,\ a^{(j_{m,1})}_{k}<p$ have to hold. But by assumption $\max t_1=j_{1,n}\leq j_{m,1}=\min t_m$ which gives that $a^{(j_{1,1})}_{k}\leq\dots\leq a^{(j_{1,n})}_{k}\leq a^{(j_{m,1})}_{k}\leq\dots\leq a^{(j_{m,n})}_{k}$ which is a contradiction.
\end{proof}

\begin{prp}\label{pmpkt}Let $K$ be an $n$-variable mean and let $T$ be admissible for $(n,m)$, $a^{(1)}\leq\dots\leq a^{(m)}$ and $a<b$. Then $K^{(T)}$ has the following properties:
\begin{enumerate}

\item[(1)] $K^{(T)}(a,\dots,a)=a$

\item[(2a)] $K(a^{(1)},\dots,a^{(n)})\leq K^{(T)}(a^{(1)},\dots,a^{(m)})\leq K(a^{(m-n+1)},\dots,a^{(m)})$

\item[(2b)] If $n=2,m=3$ then $a\leq K(a,K(a,b))\leq K^{(T)}(a,a,b)\leq K(a,b)\leq K^{(T)}(a,b,b)\leq K(K(a,b),b) \leq b$. If $K$ is strictly internal then "$\leq$" can be replaced by "$<$".

\item[(3)] $\forall k\ K^{(T)}(a^{(1)}_k,\dots,a^{(m)}_k)=K^{(T)}(a^{(1)},\dots,a^{(m)})$

\item[(4)] If $a<b$ then $\exists x\in (a,b)$ such that $K^{(T)}(a,x,\dots,x,b)=x$.

\item[(5)] If $x\in(a,b),\ K^{(T)}(a,x,\dots,x,b)=x$ then \newline$K^{(T)}(a,\dots,a,b)\leq x\leq K^{(T)}(a,b,\dots,b)$.

\end{enumerate}
\end{prp}
\begin{proof} \begin{enumerate}

\item[(1)] $\forall k\ a^{(i)}_k=a$.

\item[(2a)] Obvious from the first element of the associated sequences: \newline $K(a^{(1)},\dots,a^{(n)})\leq a^{(1)}_1\leq K^{(T)}(a^{(1)},\dots,a^{(m)})\leq a^{(m)}_1\leq K(a^{(m-n+1)},\dots,a^{(m)})$.

\item[(2b)] The first inequality is obvious. For the second let us take the associated sequences for the 3-tuples $a,a,b$. We get $a^{(1)}_2=K(a,K(a,b))$ and $(a^{(1)}_n)$ being increasing gives the second inequality. For the third consider $a^{(3)}_1=K(a,b)$ and $a^{(3)}_n$ is decreasing. The rest are similar.

Showing the ''$<$'' part, it is enough to refer to \ref{tfsnqc} because $\max t_1=2\leq\min t_3=2$.

\item[(3)] If we examine the associated sequences for $a^{(1)}_k,\dots,a^{(m)}_k$ as starting points and for $a^{(1)},\dots,a^{(m)}$ then we can see that they are the same, more precisely the indexes in the first are shifted by $k$ to the indexes of the second. 

\item[(4)] Let us define $f(x)=K^{(T)}(a,x,\dots,x,b)-x$. By Theorem \ref{tktc} $f$ is continuous. Clearly $f(a)\geq 0, f(b)\leq 0$ which implies the existence of $x$ such that $f(x)=0$.

\item[(5)] By monotonity $K^{(T)}(a,\dots,a,b)\leq K^{(T)}(a,x,\dots,x,b)\leq K^{(T)}(a,b,\dots,b)$.

\end{enumerate}
\end{proof}

\begin{ex}It can happen that $\forall k\ a^{(1)}_k= K^{(T)}(a^{(1)},\dots,a^{(m)})= a^{(m)}_k$. I.e. in \ref{pmpkt} (2a) ''$\leq$'' cannot be replaced by ''$<$''.
\end{ex}
\begin{proof}Let $n=5,m=6$. Let $t_i=\{1,2,3,4,5,6\}-\{7-i\}\ (1\leq i\leq 6)$ and $T=\{t_i:1\leq i\leq 6\}$. Clearly $T$ is admissible (the only such for (5,6)).

Let \[K(a^{(1)},\dots,a^{(5)})=\frac{\min\{a^{(1)},\dots,a^{(5)}\}+\max\{a^{(1)},\dots,a^{(5)}\}}{2}.\] 
Evidently $K$ is strictly internal, monotone and continuous.

Let $a^{(1)}=1,a^{(2)}=1,a^{(3)}=2,a^{(4)}=3,a^{(5)}=4,a^{(6)}=4$. Then $a^{(1)}_1=\dots=a^{(6)}_1=K^{(T)}(a^{(1)},\dots,a^{(6)})=\frac{1+4}{2}$.
\end{proof}

We state a theorem regarding equivalent means. We recall the classic definition.
\begin{df}Two means $K$ and $L$ are equivalent if there is a homeomorphism $f$ of $\mathbb{R}$ (or between the domains of $K,L$) such that $L=K_f$ where $K_f(a_1,\dots,a_n)=f^{-1}\big(K(f(a_1),\dots,f(a_n))\big)$.
\end{df}

\begin{thm}\label{tee}Let $T$ be an admissible system for $n<m$. Let two n-variable means $K,L$ are equivalent by function $f$. Then $K^{(T)},L^{(T)}$ are equivalent means as well and the same function $f$ testifies that. 
\end{thm}
\begin{proof} Let $a^{(1)},\dots,a^{(m)}\in\mathbb{R}$ be given. Let us create the associated sequences to $L$.

Let $a^{(1)}_0=a^{(1)},\dots,a^{(m)}_0=a^{(m)}$.

Set $a^{(i)}_{k+1}=f^{-1}\big(K(f(a^{(j_{i,1})}_k),\dots,f(a^{(j_{i,n})}_k))\big)\ (1\leq i\leq m,k\in\mathbb{N})$ where $j_{i,1},\dots,j_{i,n}\in I_m$ depend on $i$ only.

Let us investigate these sequences: $(b^{(i)})$ where 

$b^{(1)}_0=f(a^{(1)}),\dots,b^{(m)}_0=f(a^{(m)}),$

$b^{(i)}_k=f(a^{(i)}_k)\ (1\leq i\leq m)$.

If we run the same process for $K$ and $b^{(1)}_0,\dots,b^{(m)}_0$ then we end up with $K^{(T)}(b^{(1)}_0,\dots,b^{(m)}_0)$ that equals to $K^{(T)}(f(a^{(1)}),\dots,f(a^{(m)}))$. I.e. $\lim\limits_{k\to\infty}f(a^{(i)}_k)=K^{(T)}(f(a^{(1)}),\dots,f(a^{(m)}))$ or $\lim\limits_{k\to\infty}a^{(i)}_k=f^{-1}\big(K^{(T)}(f(a^{(1)}),\dots,f(a^{(m)}))\big)$ but this limit gives $L^{(T)}(a^{(1)},\dots,a^{(m)})$ as well.
\end{proof}

We close this section with some small statements regarding the cases $n=2,m=3$ and $n=2,m=4$.

\begin{prp}Suppose $K$ is a 2-variable, round mean, $a\leq b,\ k=K(a,b)$. Then $K^{(T_{2,3})}(a,k,b)=K^{(T_{2,4})}(a,k,k,b)=k$.
\end{prp}
\begin{proof}$(2,3)$: For the associated sequences $a^{(1)}_1=K(a,k),a^{(2)}_1=k,a^{(3)}_1=K(k,b)$. By roundness we get that $a^{(2)}_2=K(a^{(1)}_1,a^{(3)}_1)=k$. 
If we apply this for $a^{(1)}_1,k,a^{(3)}_1$ we get that $a^{(2)}_3=k$. By induction $\forall n\ a^{(2)}_n=k$ and $a^{(2)}_n\to K^{(T_{2,3})}(a,k,b)$. 

$(2,4)$: By the definition of the usual sequences ($a^{(1)}_0=a,a^{(2)}_0=a^{(3)}_0=k,a^{(4)}_0=b$) we get that $a^{(1)}_1=a^{(2)}_1=K(a,k),a^{(3)}_1=a^{(4)}_1=K(k,b)$ and $a^{(1)}_2=K(a,k),a^{(2)}_2=a^{(3)}_2=k,a^{(4)}_2=K(k,b)$ by roundness. From this point we can go by induction and get that $\forall n\ a^{(2)}_{2n}=k$ hence $a^{(2)}_n\to k$.
\end{proof}

\begin{prp}If $a\leq b,\ K^{(T_{2,3})}(a,x,b)=x$ implies $x=K(a,b)$ then $K$ is round.
\end{prp}
\begin{proof} Let $k=K(a,b)$. When we create the associated sequences for $a,k,b$ then $a^{(2)}_1=k$ holds. By Proposition \ref{pmpkt} (3) $K^{(T_{2,3})}(a^{(1)}_1,a^{(2)}_1,a^{(3)}_1)=K^{(T_{2,3})}(a,k,b)=k=a^{(2)}_1$. Because $a^{(2)}_2=K(a^{(1)}_1,a^{(3)}_1)$ we have $K^{(T_{2,3})}(a^{(1)}_1,a^{(2)}_2,a^{(3)}_1)=a^{(2)}_2$. By uniqueness $a^{(2)}_1=a^{(2)}_2$ that is $K$ being round.
\end{proof}



\section{Shrinking}

We descibe a generic way of reducing the number of variables of a mean that is similar the technique that we had in the previous section.

Let $K$ be a stricly internal, monotone, continuous $m$-variable mean. Let $n<m$ and $a^{(1)}\leq\dots\leq a^{(n)}\in\mathbb{R},\ (a^{(1)},\dots,a^{(n)})\in Dom\ K$ be given. We create sequences in the following way:

Let $a^{(1)}_0=a^{(1)},\dots,a^{(n)}_0=a^{(n)}$.

Set $a^{(i)}_{k+1}=K\big(a^{(1)}_k,\dots,a^{(i-1)}_k,a^{(i)}_k,\dots,a^{(i)}_k,a^{(i+1)}_k,\dots,a^{(n)}_k\big)\ (1\leq i\leq n)$ where in the middle there are $(m-n+1)$ pieces of $a^{(i)}_k$.

Therefore the associated defining system $T=T_{m,n}$ is the following: $T=\{t_1,\dots,t_n\}$ where $t_i\in I^m_n\ (1\leq i\leq n)$, $t_i=(1,\dots,i-1,i,\dots,i,i+1,\dots,n)$ and there are $(m-n+1)$ pieces of $i$ in $t_i$.

For these we can prove all previous statements:
\begin{prp}$T=T_{m,n}$ is admissible.
\end{prp}
\begin{proof}All four properties obviously hold.
\end{proof}

\begin{cor}If $a^{(1)}\leq\dots\leq a^{(n)}$ is given then $a^{(1)}\leq a^{(1)}_k\leq\dots\leq a^{(n)}_k\leq a^{(n)}$  and all associated sequences $(a^{(i)}_k)\ (1\leq i\leq n)$ converges to the same limit.
\qed
\end{cor}

\begin{df}Let us denote the common limit point by $K^{(T)}(a^{(1)},\dots,a^{(n)})$.
\end{df}

\begin{cor}$K^{(T)}$ is stricly internal, monotone, continuous n-variable mean.
\qed
\end{cor}

\begin{thm}For a quasi-arithmetic m-mean $K$, $K^{(T_{m,n})}$ is the associated n-variable quasi-arithmetic n-mean $(n<m)$.
\end{thm}
\begin{proof} We follow the proof of Theorem \ref{th7} i.e. we use the theory of Markov chains.

If $K$ is quasi-arithmetic than there is a strictly monotone, continuous function $f$ such that 
\[K(b^{(1)},\dots,b^{(m)})=f^{-1}\Big(\frac{f(b^{(1)})+\dots+f(b^{(m)})}{m}\Big).\]

Let $a^{(1)},\dots,a^{(n)}$ be given.

In exactly the same way as in \ref{th7} one can show that there exist coefficients $s^{i,k}_l\ (1\leq i,l\leq n,\ k\in\mathbb{N})$ such that 
\[a^{(i)}_{k}=f^{-1}\big(s^{i,k}_1 f(a^{(1)})+\dots+s^{i,k}_n f(a^{(n)})\big)\]
and $0\leq s^{i,k}_l\leq 1$, $\forall i\forall k\ s^{i,k}_1+\dots+s^{i,k}_m=1$ holds. Clearly 
\[s^{i,1}_l=
\begin{cases}
\frac{1}{m} & \text{if } i\ne j\\
\frac{m-n+1}{m} & \text{if } i=j
\end{cases}\]

In this case the associated stohastic $n\times n$ matrix M is 

\[M_{i,j}=\begin{cases}
\frac{1}{m} & \text{if } i\ne j\\
\frac{m-n+1}{m} & \text{if } i=j
\end{cases}\]

Similarly to Theorem \ref{th7} it can be shown that $s^{i,k}_j=(M^k)_{i,j}$. 

For $M$ it can be proved that it is irreducible, aperiodic, positive recurrent and doubly stohastic because for showing that we just need properties (2),(3) and (4) of $T$ (see Theorem \ref{th7}).
Therefore it provides a uniform limit distribution i.e. $\lim\limits_{k\to\infty}s^{i,k}_j=\frac{1}{n}$. Using the continuity of $f^{-1}$ we get the statement.
\end{proof}

We can formulate a similar statement to Theorem \ref{tm2}. We omit the proof as it is similar.
\begin{thm}\label{tm3}If $n<m,\ K_1,K_2$ are m-variable quasi-arithmetic means and $K_1\leq K_2$ holds in m variables then $K_1\leq K_2$ holds in  n variables as well.
\qed
\end{thm}

We just formulate the corresponding theorem on shrinking of equivalent means since the proof is the same (see Theorem \ref{tee}).

\begin{thm}Let two m-variable means $K,L$ are equivalent by function $f$. Then $K^{(T_{m,n})},L^{(T_{m,n})}$ are equivalent means as well and the same function $f$ testifies that ($n<m$). 
\qed
\end{thm}

\subsection{Other ways of shrinking}
For shrinking means there are many other ways as well, we provide two more.

\begin{df}If $K$ is a n-variable strictly internal and continuous mean, $a<b$ then let $K^{(s_1)}(a,b)=\inf\{x\in(a,b) : K(a,x,\dots,x,b)=x\}$.
\end{df}

The definition makes sense because \ref{pmpkt} gives that the above set is not empty. We remark that the infimum is a minimum because of continuity of $K$. Similar type of means (and shrinking) are extensively examined in \cite{kp}.

\begin{prp}The definition of $K^{(s_1)}$ provides a strictly internal, monotone and lower semi continuous mean.
\end{prp}
\begin{proof}Strict internality comes from the facts that the infimum is a minimum and $K(a,a,\dots,a,b)=a$ cannot hold.

For monotonicity let $a\leq a',b\leq b',\ K^{(s_1)}(a,b)=l,K^{(s_1)}(a',b')=l'$ i.e. $K(a,l,\dots,l,b)=l,K(a',l',\dots,l',b')=l'$. Suppose that $l'<l$. Then $K(a,l',\dots,l',b)=l'$ would contradict to $K^{(s_1)}(a,b)=l$. Hence $K(a,l',\dots,l',b)<l'$ has to hold because $K$ being monotone implies that $K(a,l',\dots,l',b)\leq K(a',l',\dots,l',b')$. Set $f(x)=K(a,x,\dots,x,b)-x$. Then $f(a)>0,f(l')<0$ therefore there is $x\in(a,l')$ such that $f(x)=0$ which is a contradiction.

 Let $K^{(s_1)}(a,b)=p,\ a_n\to a,\ b_n\to b,\ K^{(s_1)}(a_n,b_n)=l_n$. If $l_n\to l$ then $K(a_n,l_n,\dots,l_n,b_n)\to K(a,l,\dots,l,b)$ implies that $K(a,l,\dots,l,b)=l$ that gives $K^{(s_1)}(a,b)\leq l$ i.e. $K^{(s_1)}$ is lower semi continuous.
\end{proof}


We provide one more way of shrinking.

\begin{prp}(1) If $a<b$ are given, $K$ is a 2n-variable  strictly internal, monotone and continuous mean then $K^{(s_2)}(a,b)=K(a,\dots,a,b,\dots,b)$ is strictly internal, monotone, continuous where there are n pieces of a and n-pieces of b inside. 

(2) Similarly $K^{(s_3)}(a^{(1)},\dots,a^{(n)})=K(a^{(1)},\dots,a^{(n)},a^{(1)},\dots,a^{(n)})$ is strictly internal, monotone, continuous.
\qed
\end{prp}

\begin{prp}Suppose $K$ is a 2-variable, round mean. If we construct $K^{(T_{2,4})}$ then $K^{(T_{2,4})}(a,a,b,b)=K(a,b)$.
\end{prp}
\begin{proof} Let $k=K(a,b)$. By the definition of the usual sequences ($a^{(1)}_0=a^{(2)}_0=a,a^{(3)}_0=a^{(4)}_0=b$) we get that $a^{(1)}_1=a,a^{(2)}_1=a^{(3)}_1=k,a^{(4)}_1=b$ and $a^{(1)}_2=a^{(2)}_2=K(a,k),a^{(3)}_2=a^{(4)}_2=K(k,b)$. By roundness we have $a^{(2)}_3=a^{(3)}_3=k$. From this point we can go on by induction and get that $\forall n\ a^{(3)}_{2n+1}=k$ hence $a^{(3)}_n\to k$.
\end{proof}

\begin{prp}If $K$ is a 2n-variable quasi-arithmetic mean then $K^{(s_2)}$ is the corresponding 2-variable quasi-arithmetic mean and similarly $K^{(s_3)}$ is the corresponding n-variable quasi-arithmetic mean.
\end{prp}
\begin{proof}$K^{(s_2)}(a,b)=f^{-1}\big(\frac{n\cdot f(a)+n\cdot f(b)}{2n}\big)=f^{-1}\big(\frac{f(a)+f(b)}{2}\big)$.

$K^{(s_3)}(a^{(1)},\dots,a^{(n)})=f^{-1}\big(\frac{2f(a^{(1)})+\dots+2f(a^{(n)})}{2n}\big)=f^{-1}\big(\frac{f(a^{(1)})+\dots+f(a^{(n)})}{n}\big)$.
\end{proof}

We close this section with some counterexamples.

\begin{ex}$K^{(s_1)}\ne K^{(s_2)}$ in general.
\end{ex}
\begin{proof}Let $K(a,b,c,d)=\frac{\sqrt{ab}+\sqrt{ac}+\sqrt{ad}+\sqrt{bc}+\sqrt{bd}+\sqrt{cd}}{6}$. Then $K$ is strictly internal, monotone and continuous.

When calculating $K^{(s_1)}$, we have to solve $\sqrt{ab}+2\sqrt{ax}+2\sqrt{bx}+x=6x$ for $x$. The solution is $x=\Big(\frac{\sqrt{a}+\sqrt{b}+\sqrt{a+b+7\sqrt{ab}}}{5}\Big)^2$.

However for $K^{(s_2)}(a,b)$ we get $\frac{a+b+4\sqrt{ab}}{6}$.
\end{proof}

\begin{ex}Let $K(a,b,c,d)=\sqrt{ab+ac+bd+cd \over{4}}$. Then $K$ is strictly internal, monotone and continuous. An easy calculation shows that

(1) $K^{(s_1)}(a,b)=K^{(s_2)}(a,b)=\frac{a+b}{2}$.

(2) When we use the general shrinking method $K^{(T_{4,2})}$ for e.g. $a=1,b=3$ then we get $b_4<2=\frac{a+b}{2}$ hence the limit will be less than $\frac{a+b}{2}$ because $(b_n)$ is decreasing.
\end{ex}

\begin{ex}Let $K(a,b,c)=\sqrt{ab+ac+bc \over{3}}$. Then $K$ is strictly internal, monotone and continuous. An easy calculation shows that

(1) $K^{(s_1)}(a,b)=\frac{a+b+\sqrt{(a+b)^2+12ab}}{6}$.

(2) When we use the general shrinking method $K^{(T_{3,2})}$ it gives a different result, for e.g. $a=0.1,b=2$ then we get $b_3<0.781$ hence the limit will be less than that because $(b_n)$ is decreasing. And for the same values the method in (1) gives approx. $0.784$.
\end{ex}

\begin{ex}Let $L(a,b,c)=\frac{\min\{a,b,c\}+\max\{a,b,c\}}{2}$. Then 

(a) there does not exist a 2-variable strictly internal, monotone, continuous mean $K$ such that $L=K^{(T_{2,3})}$ 

(b) $(L^{(T_{3,2})})^{(T_{2,3})}\ne L$. 
\end{ex}
\begin{proof}First note that $L$ is symmetric, strictly internal, monotone and continuous hence our method is applicable. 

(a) Suppose there is such $K$. By \ref{pmpkt} (2b) we have $K^{(T_{2,3})}(a,a,b)<K^{(T_{2,3})}(a,b,b)\ (a<b)$ which obviously does not hold for $L$. 

(b) Obviously $L^{(T_{3,2})}(a,c)=\frac{a+c}{2}$ because all sequences are equal to that value. But this is the 2-variable arithmetic mean and then $(L^{(T_{3,2})})^{(T_{2,3})}\ne L$ since $(L^{(T_{3,2})})^{(T_{2,3})}$ is the 3-variable arithmetic mean. 
\end{proof}

\section{On compounding}
We can simply generalize our extension method from $n$-variable to $m$-variable by interchanging $K$ to $m$ pieces of $n$-variable means.

\begin{thm}Let $n<m$ and $K_1,\dots, K_m$ $n$-variable means be given such that $K_1\leq\dots\leq K_m$.  Let $a^{(1)}\leq\dots\leq a^{(m)}$ and $T$ be an admissible system for $(n,m)$. Let us define $m$ sequences in the following way.

Let $a^{(1)}_0=a^{(1)},\dots,a^{(m)}_0=a^{(m)}$ and set $a^{(i)}_{k+1}=K_i(a^{(j_{i,1})}_k,\dots,a^{(j_{i,n})}_k)\ (1\leq i\leq m,k\in\mathbb{N})$ where $t_i\in T,\ t_i=(j_{i,1},\dots,j_{i,n})\in I^n_m$.

Then all sequences converge to the same limit that is between $a^{(1)}$ and $a^{(m)}$. If we consider it as a mean of $a^{(1)},\dots,a^{(m)}$ then this mean is strictly internal, monotone and continuous.
\end{thm}
\begin{proof}For convergence replace $K$ by $K_j$ in the proof of \ref{t4}.

For showing the second part, copy the proof of \ref{tktc} substituting $K$ by $K_1,\dots, K_m$ and remark that \ref{r2b} remains valid as well.
\end{proof}

If $n=m$ and $\forall i\ j_{i,h}=h$ then clearly it is a generalization of compounding of two means.

\section{Symmetrization}
Using similar technique we can symmetrize a non-symmetric 2-variable mean. Let $\circ$ be a non-symmetric, strictly internal, monotone, continuous mean.
Let $a<b \in\mathbb{R}$ be given. Let us define two sequences:

$a_0=a,b_0=b$.

$a_{n+1}=\min{\{a_n\circ b_n,b_n\circ a_n\}},b_{n+1}=\max{\{a_n\circ b_n,b_n\circ a_n\}}$.

Obviously $a\leq a_n$ is increasing while $b_n\leq b$ is decreasing, therefore both converges. By continuity they must converge to the same limit point. Let us denote it by $K^{(sym)}(a,b)$.

\begin{prp}$K^{(sym)}$ is symmetric.
\end{prp}
\begin{proof} The associated sequences for $(a,b)$ and $(b,a)$ are the same.
\end{proof}


{\footnotesize

\smallskip
\noindent
Dennis G\'abor College, Hungary 1119 Budapest Fej\'er Lip\'ot u. 70.

\noindent 
E-mail: losonczi@gdf.hu, alosonczi1@gmail.com\\
}

\end{document}